\documentclass[12pt]{amsart}

\title{Maximum Waring ranks of monomials}

\author{Erik Holmes}
\address{Erik Holmes \\
Department of Mathematics \\
1910 University Drive \\
Boise State University \\
Boise, ID 83725-1555 \\
USA}
\email{erikholmes@u.boisestate.edu}

\author{Paul Plummer}
\address{Paul Plummer \\
Department of Mathematics \\
1910 University Drive \\
Boise State University \\
Boise, ID 83725-1555 \\
USA}
\email{paulplummer@u.boisestate.edu}

\author{Jeremy Siegert}
\address{Jeremy Siegert \\
Department of Mathematics \\
1910 University Drive \\
Boise State University \\
Boise, ID 83725-1555 \\
USA}
\email{jeremysiegert@u.boisestate.edu}

\author{Zach Teitler}
\address{Zach Teitler \\
Department of Mathematics \\
1910 University Drive \\
Boise State University \\
Boise, ID 83725-1555 \\
USA}
\email{zteitler@boisestate.edu}

\date{\today}



\usepackage{amssymb}
\usepackage[letterpaper,margin=1in]{geometry}
\usepackage{booktabs}
\usepackage{hyperref}

\newtheorem{theorem}{Theorem}
\newtheorem{proposition}[theorem]{Proposition}
\newtheorem{lemma}[theorem]{Lemma}

\theoremstyle{definition}
\newtheorem{example}[theorem]{Example}

\newcommand{\rgen}{r_{\mathrm{gen}}}
\newcommand{\rmax}{r_{\mathrm{max}}}

\newcommand{\defining}[1]{\textbf{#1}}

\begin{document}

\begin{abstract}
We show that monomials and sums of pairwise coprime monomials
in four or more variables have Waring rank less than the generic rank,
with a short list of exceptions.
We asymptotically compare their ranks with the generic rank.
\end{abstract}

\maketitle

\section{Introduction}

Let $F(x_1,\dotsc,x_n)$ be a polynomial in the variables $x_1,\dotsc,x_n$ with complex coefficients.
We assume $F$ is homogeneous of degree $d$.
The \defining{Waring rank} of $F$, denoted $r(F)$, is the least number of terms needed to write $F$ as a linear combination
of $d$th powers of linear polynomials, $F = c_1 \ell_1^d + \dotsb + c_r \ell_r^d$.
For example, $F(x,y) = xy$ can be written as
\[
  xy = \frac{1}{4}(x+y)^2 - \frac{1}{4}(x-y)^2
\]
which shows $r(xy) \leq 2$.
On the other hand we must have $r(xy) > 1$ because if $xy = c_1 \ell_1^2$ then $xy$ would be a perfect square,
which it is not.
Therefore $r(xy) = 2$.
Similarly,
\[
  xyz = \frac{1}{24} \Big( (x+y+z)^3 - (x+y-z)^3 - (x-y+z)^3 + (x-y-z)^3 \Big)
\]
which shows $r(xyz) \leq 4$.
It turns out that $r(xyz) = 4$, but this is not obvious.
See for example \cite{Landsberg:2009yq,MR2842085} for proofs.

Waring ranks of homogeneous forms have been studied since the 19th century
by Sylvester and others.
For modern introductions see for example \cite{MR1735271,MR2865915,Reznick:2013uq}.
For numerous applications in engineering, sciences, and other areas of mathematics,
see for example \cite{comonmour96,MR2865915}.

It is surprisingly difficult to determine the Waring rank $r(F)$ for an arbitrary homogeneous polynomial
of degree $d$ in $n$ variables, henceforth called a \defining{$d$-form}.
Some cases are known.
For example, in $n=2$ variables, $r(F)$ can be determined by Sylvester's results in \cite{Sylvester:1851kx,Sylvester:1851wd};
for more recent treatments see for example \cite{MR859177,MR2754189,Reznick:2013uq}.
For another example, in degree $d=2$, $F$ is a quadratic form, which can be represented by a symmetric matrix;
then the Waring rank $r(F)$ is equal to the rank of this matrix.
One way to think of Waring rank is as a higher-degree generalization of matrix rank (for symmetric matrices).

It is also difficult to determine the maximum Waring rank occuring for $d$-forms in $n$ variables.
Here is one simple upper bound:
the space of $d$-forms in $n$ variables is spanned by the powers of linear forms,
so choosing a basis consisting of powers of linear forms shows that
the Waring rank $r(F)$ is at most the dimension of the space, $r(F) \leq \binom{d+n-1}{n-1}$.

This was improved to $r(F) \leq \binom{d+n-2}{n-1}$ in \cite{MR2383331},
then $r(F) \leq \binom{d+n-2}{n-1} - \binom{d+n-6}{n-3}$ in \cite{Jelisiejew:2013fk},
and then $r(F) \leq \binom{d+n-2}{n-1} - \binom{d+n-6}{n-3} - \binom{d+n-7}{n-3}$ in \cite[Prop.~3.9]{Ballico:2013sf}.
See also \cite{Blekherman:2014eq}.
But the actual maximum rank is known in only a few cases.
Binary ($n=2$) forms of degree $d$ have rank at most $d$, with $r(xy^{d-1})=d$.
For quadratic forms ($d=2$) the maximum rank is $n$.
For $(n,d)=(3,3)$ Jelisiejew's and Ballico--De Paris's upper bounds are $5$
and it is known that there are forms of rank $5$ \cite{MR1506892,comonmour96,Landsberg:2009yq}.
Finally, for $(n,d)=(3,4)$ Jelisiejew's upper bound is $9$,
Ballico--De Paris's upper bound is $8$,
and it is known that there are forms of rank $7$.
It has been shown that in this case the maximum rank is actually $7$ \cite{Kleppe:1999fk,Paris:2013fk}.
So the search for a sharp upper bound continues.


For each value of $n$ and $d$
there is value of rank that holds for all forms in a dense Zariski open subset of the space of forms,
called the \defining{generic} Waring rank of a form in $n$ variables of degree $d$.
We denote it by $\rgen(n,d)$.
By the Alexander--Hirschowitz theorem
\cite{MR1311347} 
it is given by
\begin{equation}\label{eq: generic rank}
  \rgen(n,d) = \left\lceil \frac{1}{n} \binom{d+n-1}{n-1} \right\rceil,
\end{equation}
except if $(n,d) = (n,2), (3,4), (4,4), (5,3), (5,4)$.
In the exceptional cases $\rgen(n,2)=n$ (instead of $\lceil (n+1)/2 \rceil$),
$\rgen(3,4)=6$ (instead of $5$),
$\rgen(4,4)=10$ (instead of $9$),
$\rgen(5,3)=8$ (instead of $7$),
and $\rgen(5,4)=15$ (instead of $14$).

Clearly the maximum value of rank is at least the generic value, and at most the upper bound of Jelisiejew
or Ballico--De Paris.
The gap between the generic rank and these upper bounds is not too large:
\[
  \frac{\binom{d+n-2}{n-1} - \binom{d+n-6}{n-3} - \binom{d+n-7}{n-3}}{\left\lceil \frac{1}{n} \binom{d+n-1}{n-1} \right\rceil}
  \lessapprox
  \frac{\binom{d+n-2}{n-1}}{\frac{1}{n}\binom{d+n-1}{n-1}}
  =
  \frac{dn}{d+n-1} .
\]
This has been improved in \cite{Blekherman:2014eq}, where it is shown that the maximum value of rank
is at most twice the generic value.
We can narrow this gap further either by finding new upper bounds
or by finding forms with greater than generic rank.

It turns out that very few examples are known of forms with greater than generic rank,
with $n \geq 3$ variables.
(Plenty are known for $n=2$.)
In fact, it seems that until recently only finitely many such examples were known:
just some cubics and quartics ($d=3,4$) in $n=3$ variables.

Recently, however, an infinite family of forms was discovered to have greater than generic rank.
This family was found by
Carlini, Catalisano, and Geramita in their solution of the Waring rank problem for monomials \cite{Carlini20125}.
Let $M = x_1^{a_1} \dotsm x_n^{a_n}$ with $0 < a_1 \leq \dotsb \leq a_n$.
They showed that the Waring rank of $M$ is
\[
  r(M) = (a_2+1)\dotsm(a_n+1) .
\]
The rank of $M$ is maximized when $a_1 = 1$ and the remaining exponents $a_2,\dotsc,a_n$
are as close as possible to being equal.
Explicitly, let $d = a_1 + \dotsb + a_n = \deg M$
and write $d-1 = q(n-1) + s$ with $0 \leq s < n-1$.
Then the maximum rank monomial in $n$ variables of degree $d$ is
\[
  x_1 x_2^q \dotsm x_{n-s}^q x_{n-s+1}^{q+1} \dotsm x_n^{q+1} .
\]
This has rank
\[
  (1+q)^{n-s-1} (2+q)^s
    = \left(1 + \left\lfloor \frac{d-1}{n-1} \right\rfloor\right)^{n-s-1} \left(1 + \left\lceil \frac{d-1}{n-1} \right\rceil\right)^{s}
  \approx
  \left(1 + \frac{d-1}{n-1}\right)^{n-1} .
\]
So for fixed $n$, asymptotically in $d$ the maximum rank of monomials is $d^{n-1}/(n-1)^{n-1}$,
while by \eqref{eq: generic rank} asymptotically the generic rank is $d^{n-1}/n!$.
For $n=3$, therefore, the maximum rank of monomials is asymptotically $3/2$ the generic rank.
In particular there are infinitely many monomials in three variables with greater than generic rank,
and in fact it is easy to see that they occur in every degree $d \geq 5$.
On the other hand for $n \geq 4$ we have $(n-1)^{n-1} > n!$,
so for $d \gg 0$, monomials have less than generic rank.
This shows that for each $n \geq 4$, there are (at most) finitely many monomials
with higher than generic rank.
All of this was observed in \cite{Carlini20125}.

We show here that in fact, in four or more variables there are \emph{absolutely no} monomials with higher than generic rank.
Then we consider sums of pairwise coprime monomials,
whose ranks were also determined in \cite{Carlini20125}.
We show that all such sums have less than generic rank, with exactly three exceptions.
Next we asymptotically compare the maximum ranks of monomials and sums of pairwise coprime monomials with the generic rank.
Finally we briefly discuss non-monomial examples.

\section{Ranks of monomials in four or more variables}

\emph{All} monomials in four or more variables have less than generic rank.

\begin{theorem}\label{thm: monomial}
Let $M$ be a monomial in $n \geq 4$ variables
and let $d = \deg M > 1$.
Then $r(M) < \rgen(n,d)$.
\end{theorem}

\noindent
That is, the only monomials with greater than generic rank are in three or fewer variables.

We do not assume that $M$ actually involves every variable.

\begin{proof}
We assume $n \geq 4$, so we ignore the exceptional case $(n,d) = (3,4)$.
First we dispose of the remaining exceptional cases.
If $d=2$ then $M = x_1^2$ or $x_1 x_2$, so $r(M) = 1$ or $2$, while $\rgen(n,2) = n \geq 4 > r(M)$.
The cases $(n,d) = (4,4), (5,3)$ are listed in Table~\ref{table: exceptional cases}.
For $(n,d) = (5,4)$, the same monomials appear as in the case $(n,d)=(4,4)$, and $\rgen(5,4) = 15$.
This takes care of all the exceptional cases.

\begin{table}[hb]
\begin{tabular}{lllll}
\toprule
$n$ & $d$ & $M$ & $r(M)$ & generic rank \\
\midrule
$4$ & $4$ & $x_1 x_2 x_3 x_4$ & $8$ & $10$ \\
  &  & $x_1 x_2 x_3^2$ & $6$ & $10$ \\
  &  & $x_1 x_2^3$ & $4$ & $10$ \\
  &  & $x_1^2 x_2^2$ & $3$ & $10$ \\
  &  & $x_1^4$ & $1$ & $10$ \\
$5$ & $3$ & $x_1 x_2 x_3$ & $4$ & $8$ \\
  &  & $x_1 x_2^2$ & $3$ & $8$ \\
  &  & $x_1^3$ & $1$ & $8$ \\
\bottomrule
\end{tabular}
\caption{Exceptional cases $(4,4)$ and $(5,3)$.}\label{table: exceptional cases}
\end{table}

Now we consider the nonexceptional case.
Say $k \leq n$ of the variables appear in $M = x_1^{a_1} \dotsm x_k^{a_k}$, $0 < a_1 \leq \dotsb \leq a_k$.
Write $M = x_1^{a_1} \dotsm x_n^{a_n}$, $a_{k+1} = \dotsb = a_n = 0$.
By the arithmetic-geometric mean inequality,
\begin{multline*}
  r(M) = (a_2+1)\dotsm(a_n+1)
    \leq \left( \frac{a_2+\dotsb+a_n+n-1}{n-1} \right)^{n-1} \\
    = \left( \frac{d+n-1-a_1}{n-1} \right)^{n-1}
    \leq \left( \frac{d+n-2}{n-1} \right)^{n-1} .
\end{multline*}
We claim
\[
  \left( \frac{d+n-2}{n-1} \right)^{n-1} < \frac{1}{n} \binom{d+n-1}{n-1}
\]
or equivalently
\begin{equation}\label{eq: inequality}
  \left(\frac{d+n-2}{d+n-1}\right)
      \dotsm \left(\frac{d+n-2}{d+1}\right)
  <
  \frac{(n-1)^{n-1}}{n!} \\
  =
  \left(\frac{n-1}{n}\right)
      \dotsm \left(\frac{n-1}{2}\right) .
\end{equation}
First, if $m, c > 0$ then $(c+m)^2(m^2-1) < m^2((c+m)^2-1)$,
so
\[
  \frac{(c+m)^2}{(c+m+1)(c+m-1)} < \frac{m^2}{(m+1)(m-1)} .
\]
Substituting $c=d-1$ and $m=n-1$,
\[
  \left(\frac{d+n-2}{d+n-1}\right) \left(\frac{d+n-2}{d+n-3}\right)
    <
  \left(\frac{n-1}{n}\right) \left(\frac{n-1}{n-2}\right) .
\]
This takes care of the first three factors on each side in \eqref{eq: inequality}.
(Here we use the hypothesis $n\geq 4$; otherwise $\frac{d+n-2}{d+n-3}$ and $\frac{n-1}{n-2}$ are absent.)
For the remaining factors,
\[
  \frac{(d-1)+(n-1)}{(d-1)+a} < \frac{n-1}{a}
\]
for $2 \leq a < n-1$.
This proves \eqref{eq: inequality} and completes the proof.
\end{proof}

\section{Sums of pairwise coprime monomials}

If $M_1,\dotsc,M_t$ are pairwise coprime monomials, that is, involving pairwise disjoint sets of variables,
then $r(M_1 + \dotsb + M_t) = \sum r(M_i)$ \cite{Carlini20125}.

\begin{example}
The form $F = x_1 x_2^2 + x_3 x_4^2$, with $n=4$, $d=3$, has higher than generic rank:
\[
  r(x_1 x_2^2 + x_3 x_4^2) = r(x_1 x_2^2) + r(x_3 x_4^2) = 6 > \rgen(4,3) = 5 .
\]
The forms $x_1 x_2 x_3 + x_4^3$ and $x_1 x_2^2 + x_3^3 + x_4^3$ each have rank $5$,
equal to the generic rank.
\end{example}

In fact these are the only examples with greater than or equal to generic rank.

\begin{theorem}\label{thm: no high rank sums of coprime}
Every sum of pairwise coprime monomials in $n \geq 4$ variables, of degree $d \geq 3$,
has rank strictly less than the generic rank, with the following list of exceptions
all occurring in $(n,d) = (4,3)$:
$x_1 x_2^2 + x_3 x_4^2$ has rank $6$, strictly greater than $\rgen(4,3) = 5$;
$x_1 x_2 x_3 + x_4^3$ and $x_1 x_2^2 + x_3^3 + x_4^3$ have rank $5$, equal to the generic rank.
\end{theorem}

Let $\rmax(n,d)$ be the maximum rank of a monomial of degree $d$ in $n$ variables
and let $\rmax^*(n,d)$ be the maximum rank of a sum of pairwise coprime monomials of degree $d$ in $n$ variables.

\begin{lemma}\label{lemma: slope inequality}
If $d \geq 4$ and $d \geq n \geq 2$,
then $\frac{1}{n} \rmax(n,d) \geq \frac{1}{n-1} \rmax(n-1,d)$.
\end{lemma}

\begin{proof}
First, if $n=2$, $\rmax(2,d) = d$ while $\rmax(1,d) = 1$, so the claim is true.
Second, suppose $d > n > 2$.
Let $M$ be a monomial in $n-1$ variables of degree $d$ with rank $\rmax(n-1,d)$.
Up to reordering the variables, $M = x_1^{1} x_2^{a_2} \dotsm x_{n-1}^{a_{n-1}}$
with $a_2 \leq \dotsb \leq a_{n-1} \leq a_2 + 1$.
We have
\[
  a_{n-1} = \left\lceil \frac{d-1}{n-2} \right\rceil   > 1 .
\]
Let $M' = x_1^{1} x_2^{a_2} \dotsm x_{n-1}^{a_{n-1} - 1} x_n^{1}$, so $M'$ still has degree $d$,
and
\begin{equation}\label{eq: rmax slope increasing}
  \rmax(n,d) \geq r(M') = r(M) \frac{2 a_{n-1}}{a_{n-1} + 1}
    = \rmax(n-1,d) \frac{ 2 \lceil\frac{d-1}{n-2}\rceil }{ \lceil\frac{d-1}{n-2}\rceil + 1}
    \geq \rmax(n-1,d) \frac{2(d-1)}{d+n-3} .
\end{equation}
Since $d \geq n+1$ we have
\[
  d \geq \frac{n^2-n-2}{n-2},
\]
which after some rearrangement becomes
\[
  \frac{2(d-1)}{d+n-3} \geq \frac{n}{n-1} .
\]
Combining with \eqref{eq: rmax slope increasing}, this proves the claim for the case $d > n > 2$.

Finally, if $d = n \geq 4$ then $\rmax(n,d) = 2^{d-1}$ and $\rmax(n-1,d) = 3 \cdot 2^{d-3} = \frac{3}{4} \rmax(n,d)$.
Since $n \geq 4$, $\frac{n-1}{n} \geq \frac{3}{4}$, so $\frac{1}{n} \rmax(n,d) \geq \frac{1}{n-1} \rmax(n-1,d)$.
\end{proof}

\begin{proof}[Proof of Theorem~\ref{thm: no high rank sums of coprime}]
First suppose $d \geq n \geq 4$.
Let $F = M_1 + \dotsb + M_s$ be a sum of pairwise coprime monomials of degree $d$,
where $M_i$ involves exactly $n_i$ variables,
$n = \sum n_i$, $n_1 \geq \dotsb \geq n_s \geq 1$.
We use the elementary inequality that if $a_i, b_i > 0$ and $\frac{a_i}{b_i} \leq x$ for all $i$,
then $(\sum a_i)/(\sum b_i) \leq x$.
For each $i$,
\[
  \frac{r(M_i)}{n_i} \leq \frac{\rmax(n_i,d)}{n_i} \leq \frac{\rmax(n,d)}{n}
\]
by Lemma~\ref{lemma: slope inequality}. Therefore
\[
  \frac{r(F)}{n} = \frac{\sum r(M_i)}{\sum n_i} \leq \frac{\rmax(n,d)}{n} ,
\]
hence $r(F) \leq \rmax(n,d)$.
By Theorem~\ref{thm: monomial}, $\rmax(n,d) < \rgen(n,d)$ for $n \geq 4$.

Next we take care of the case $d = 3$, $n \geq 5$.
Let $F$ be a sum of pairwise coprime monomials of degree $3$ with rank $\rmax^*(n,3)$.
The only monomials that can appear are of the form $x^3$, $x y^2$, $xyz$,
with ranks $1$, $3$, $4$ respectively.
We can replace each occurence in $F$ of $xyz$ with $xy^2 + z^3$
without changing the rank or number of variables.
So we can assume every term in $F$ is of the form $x^3$ or $x y^2$.
This shows that if $n$ is even, $\rmax^*(n,3) = 3n/2$, and if $n$ is odd, $\rmax^*(n,3) = (3n-1)/2$.
On the other hand,
\[
  \rgen(n,3) \geq \frac{1}{n} \binom{n+2}{3} = \frac{(n+2)(n+1)}{6} = \frac{3n}{2} + \frac{n(n-6)+2}{6} .
\]
When $n \geq 6$, $n(n-6)+2 \geq 2$, which shows $\rgen(n,3) > \frac{3n}{2} \geq \rmax^*(n,3)$.
When $n = 5$, $\rgen(5,3) = 8$ (by the Alexander--Hirschowitz theorem) while $\rmax^*(5,3) = 7$.

Now we deal with the case $n > d \geq 4$.
We will use that
\begin{equation}\label{eq: gen rank vs monomial n > d}
  \frac{1}{n^2} \binom{d+n-1}{n-1} > \frac{2^{d-1}}{d} .
\end{equation}
To prove this, first, for $n = d \geq 4$ we have
\[
  \frac{1}{d^2} \binom{2d-1}{d-1}
    = \frac{1}{d} \left(\frac{2d-1}{d}\right) \left(\frac{2d-2}{d-1}\right) \dotsm \left(\frac{d+1}{2}\right) .
\]
Note that
\[
  \left(\frac{2d-1}{d}\right) \left(\frac{2d-3}{d-2}\right) = \frac{4d^2-8d+3}{d^2-2d} > 4,
\]
while the $d-3$ remaining factors are $\geq 2$.
So we have $\frac{1}{d^2} \binom{2d-1}{d-1} > \frac{2^{d-1}}{d}$.

Next, for $n \geq d \geq 4$ we have
\[
\begin{split}
  \frac{1}{n^2} \binom{d+n-1}{n-1} &= \frac{(n+1)^2}{n^2} \frac{n}{n+d} \cdot \frac{1}{(n+1)^2} \binom{d+n}{n} \\
    &= \frac{(n+1)^2}{n(n+d)} \cdot \frac{1}{(n+1)^2} \binom{d+n}{n} \\
    & < \frac{1}{(n+1)^2} \binom{d+n}{n}
\end{split}
\]
since $d > 2$.
This completes the proof of \eqref{eq: gen rank vs monomial n > d}.

Now let $F = M_1 + \dotsb + M_s$ be a sum of pairwise coprime monomials of degree $d$,
where $M_i$ involves exactly $n_i$ variables,
$n = \sum n_i$, $n_1 \geq \dotsb \geq n_s \geq 1$.
Note each $n_i \leq d$.
We have
\[
  \frac{r(M_i)}{n_i} \leq \frac{\rmax(d,d)}{d} = \frac{2^{d-1}}{d} .
\]
Therefore
\begin{equation}\label{eq: linear bound for rmax*}
  \frac{r(F)}{n} = \frac{\sum r(M_i)}{\sum n_i} \leq \frac{2^{d-1}}{d} < \frac{1}{n^2} \binom{d+n-1}{n-1}
\end{equation}
which gives us
\[
  r(F) < \frac{1}{n} \binom{d+n-1}{n-1} \leq \rgen(n,d)
\]
as desired.
This completes the case $n > d \geq 4$.

Finally, we consider the case $(n,d) = (4,3)$.
Up to reordering terms and variables,
the sums of pairwise coprime monomials that use all the variables are listed in the following table.
\[
\begin{array}{ll}
F & r(F) \\
\toprule
x_1^3 + x_2^3 + x_3^3 + x_4^3 & 4 \\
x_1 x_2^2 + x_3^3 + x_4^3 & 5 \\
x_1 x_2 x_3 + x_4^3 & 5 \\
x_1 x_2^2 + x_3 x_4^2 & 6
\end{array}
\]
Since $\rgen(4,3) = 5$, this shows that the exceptions listed in the statement of the theorem are the only ones.
\end{proof}

We have seen that if $d \geq n$ then $\rmax^*(n,d)$ is attained by a monomial.
What if $n > d$?
The greatest rank monomial of degree $d$ is a product of $d$ variables.
So a ``greedy'' way to construct a high-rank sum of pairwise coprime monomials
is to add up products of $d$ variables,
with any remaining variables placed into one more monomial.
But this does not necessarily maximize Waring rank,
as we have seen for $(n,d)=(4,3)$: the greedy choice $x_1 x_2 x_3 + x_4^3$ has rank $5$,
while the non-greedy choice $x_1 x_2^2 + x_3 x_4^2$ has rank $6$.
Similarly, for $(n,d)=(5,4)$, the greedy choice $x_1 x_2 x_3 x_4 + x_5^4$ has rank $9$,
while the non-greedy choice $x_1 x_2 x_3^2 + x_4 x_5^3$ has rank $10$;
and for $(n,d)=(6,5)$, the greedy choice $x_1 \dotsm x_5 + x_6^5$ has rank $17$,
while the non-greedy choice $x_1 x_2^2 x_3^2 + x_4 x_5^2 x_6^2$ has rank $18$.

\section{Asymptotic comparison with generic rank}\label{section: asymptotic}

It was noted in \cite{Carlini20125} that, for fixed $n$ and $d$ going to infinity,
$\rmax(n,d)$ is asymptotically $d^{n-1}/(n-1)^{n-1}$, while $\rgen(n,d)$ is asymptotically $d^{n-1}/n!$.
So $\rmax(3,d)/\rgen(3,d) \to 3/2$, while for $n > 3$,
\[
  \frac{\rmax(n,d)}{\rgen(n,d)} \to \frac{n!}{(n-1)^{n-1}} < 1
  \qquad
  \text{as $d \to \infty$}.
\]
If $d$ is fixed and $n \to \infty$, then, for $n \geq d$, $\rmax(n,d) = \rmax(d,d) = 2^{d-1}$,
since the highest rank monomial is $x_1 \dotsm x_d$, and the extra variables cannot be used.

Similarly, if $n \geq 4$ is fixed and $d \to \infty$, then, for $d \geq n$, $\rmax^*(n,d) = \rmax(n,d)$,
and once again,
\[
  \frac{\rmax^*(n,d)}{\rgen(n,d)} \to \frac{n!}{(n-1)^{n-1}} < 1
  \qquad
  \text{as $d \to \infty$}.
\]

Finally, fix $d$ to find the limit of the ratio as $n \to \infty$.
In the proof of Theorem~\ref{thm: no high rank sums of coprime} we found that,
for a fixed $d \geq 3$,  $\rmax^*(n,d)$ is bounded by a linear function for large enough $n$:
$\rmax^*(n,d) \leq \tfrac{3n}{2}$ when $d=3$ and $\rmax^*(n,d) \leq \tfrac{n2^{d-1}}{d}$ when $n > d \geq 4$,
by~\eqref{eq: linear bound for rmax*}.
However for $d \geq 3$, $\rgen(n,d) = O(n^{d-1})$ grows faster than a linear function as a function of $n$, so
\[
  \frac{\rmax^*(n,d)}{\rgen(n,d)} \to 0
  \qquad
  \text{as $n \to \infty$}.
\]

\section{Non-monomial examples}

We close with examples
of forms which are not monomials or sums of pairwise coprime monomials,
with higher than generic rank.

First, $F = x^2 y + y^2 z$ is a form in $n=3$ variables of degree $d=3$,
with $r(F) = 5 > \rgen(3,3) = 4$.
See for example \cite[\textsection8]{Landsberg:2009yq}, \cite[Theorem 2.3]{Kleppe:1999fk}.
Second, $r(x^2 y^2 + y^3 z) = 7 > \rgen(3,4) = 6$, see \cite[Proposition 3.1]{Kleppe:1999fk}.

To these we can add one more non-monomial example:
\begin{proposition}
Let $F = x^2 y + y^2 z$ be the plane cubic of rank $5$.
Let $G = F + w^3$.
Then $r(G) = 6 > \rgen(4,3) = 5$.
\end{proposition}
We thank Jarek Buczy\'nski for suggesting this example and a proof that involved tensor rank.
Since then, however, we have learned of a much quicker, elementary proof
using a very recent result of Carlini, Catalisano, and Chiantini \cite{Carlini:2014rr}.
They showed that $r(F(x_1,\dotsc,x_n) + y_1^d + \dotsb + y_s^d) = r(F) + s$,
when the $x_i$ and $y_j$ are independent variables.
This immediately implies the proposition.

They also showed that $r(F(x_1,x_2) + G(y_1,y_2)) = r(F) + r(G)$.
It is expected that this should hold for forms in any number of variables.
Unfortunately, the result for two forms in two variables does not give any new examples of forms with higher than generic rank:
$r(F+G) = r(F) + r(G)$ is maximized when $F = x_1 x_2^{d-1}$, $G = y_1 y_2^{d-1}$,
a sum of pairwise coprime monomials, already considered above.

In conclusion, it is surprisingly nontrivial not only to find forms with strictly greater than generic rank,
but even just to find forms with Waring rank equal or close to generic rank.

\section*{About the authors}
The first three authors are undergraduate students at Boise State University.
The fourth author, a faculty member at Boise State University, proved Theorem~\ref{thm: monomial}
and conjectured the statement of Theorem~\ref{thm: no high rank sums of coprime}.
The proof of that conjecture and the limits in Section~\ref{section: asymptotic}
were found by the first three authors.

\bibliographystyle{amsalpha}
\renewcommand{\MR}[1]{{}}
\bibliography{../../biblio}

\providecommand{\bysame}{\leavevmode\hbox to3em{\hrulefill}\thinspace}
\providecommand{\MR}{\relax\ifhmode\unskip\space\fi MR }
\providecommand{\MRhref}[2]{%
  \href{http://www.ams.org/mathscinet-getitem?mr=#1}{#2}
}
\providecommand{\href}[2]{#2}
\begin{thebibliography}{CCG12}

\bibitem[AH95]{MR1311347}
J.~Alexander and A.~Hirschowitz, \emph{Polynomial interpolation in several
  variables}, J. Algebraic Geom. \textbf{4} (1995), no.~2, 201--222.
  \MR{1311347 (96f:14065)}

\bibitem[BBS08]{MR2383331}
A.~Bia{\l}ynicki-Birula and A.~Schinzel, \emph{Representations of multivariate
  polynomials by sums of univariate polynomials in linear forms}, Colloq. Math.
  \textbf{112} (2008), no.~2, 201--233. \MR{2383331 (2009b:12006)}

\bibitem[BP13]{Ballico:2013sf}
Edoardo Ballico and Alessandro~De Paris, \emph{Generic power sum decompositions
  and bounds for the {W}aring rank},
  \href{http://arxiv.org/abs/1312.3494}{\nolinkurl{arXiv:1312.3494}} [math.AG],
  Dec 2013.

\bibitem[BT14]{Blekherman:2014eq}
Greg Blekherman and Zach Teitler, \emph{On maximum, typical, and generic
  ranks}, \href{http://arxiv.org/abs/1402.2371}{arXiv:1402.2371} [math.AG], Feb
  2014.

\bibitem[CCC14]{Carlini:2014rr}
Enrico Carlini, Maria~Virginia Catalisano, and Luca Chiantini, \emph{Progress
  on the symmetric {S}trassen conjecture},
  \href{http://arxiv.org/abs/1405.3721}{\nolinkurl{arXiv:1405.3721}} [math.AG],
  May 2014.

\bibitem[CCG12]{Carlini20125}
Enrico Carlini, Maria~Virginia Catalisano, and Anthony~V. Geramita, \emph{The
  solution to the {W}aring problem for monomials and the sum of coprime
  monomials}, J. Algebra \textbf{370} (2012), 5--14.

\bibitem[CM96]{comonmour96}
P.~COMON and B.~MOURRAIN, \emph{Decomposition of quantics in sums of powers of
  linear forms}, Signal Processing, Elsevier 53(2), 1996.

\bibitem[CS11]{MR2754189}
Gonzalo Comas and Malena Seiguer, \emph{On the rank of a binary form}, Found.
  Comput. Math. \textbf{11} (2011), no.~1, 65--78. \MR{2754189}

\bibitem[IK99]{MR1735271}
Anthony Iarrobino and Vassil Kanev, \emph{Power sums, {G}orenstein algebras,
  and determinantal loci}, Lecture Notes in Mathematics, vol. 1721,
  Springer-Verlag, Berlin, 1999, Appendix C by Iarrobino and Steven L. Kleiman.
  \MR{1735271 (2001d:14056)}

\bibitem[Jel13]{Jelisiejew:2013fk}
Joachim Jelisiejew, \emph{An upper bound for the {W}aring rank of a form},
  \href{http://arxiv.org/abs/1305.6957}{\nolinkurl{arXiv:1305.6957}} [math.AC],
  May 2013.

\bibitem[Kle99]{Kleppe:1999fk}
Johannes Kleppe, \emph{Representing a homogenous polynomial as a sum of powers
  of linear forms}, Master's thesis, University of Oslo, 1999,
  \url{http://folk.uio.no/johannkl/kleppe-master.pdf}.

\bibitem[Kun86]{MR859177}
Joseph P.~S. Kung, \emph{Gundelfinger's theorem on binary forms}, Stud. Appl.
  Math. \textbf{75} (1986), no.~2, 163--169. \MR{859177 (87m:11020)}

\bibitem[Lan12]{MR2865915}
J.~M. Landsberg, \emph{Tensors: geometry and applications}, Graduate Studies in
  Mathematics, vol. 128, American Mathematical Society, Providence, RI, 2012.
  \MR{2865915}

\bibitem[LT10]{Landsberg:2009yq}
J.M. Landsberg and Zach Teitler, \emph{On the ranks and border ranks of
  symmetric tensors}, Found.\ Comp.\ Math. \textbf{10} (2010), no.~3, 339--366.

\bibitem[Par13]{Paris:2013fk}
Alessandro~De Paris, \emph{A proof that the maximal rank for plane quartics is
  seven},
  \href{http://arxiv.org/abs/1309.6475}{\nolinkurl{http://arxiv.org/abs/1309.6475}},
  Sep 2013.

\bibitem[Rez13]{Reznick:2013uq}
Bruce Reznick, \emph{On the length of binary forms}, Quadratic and Higher
  Degree Forms (New York) (K.~Alladi, M.~Bhargava, D.~Savitt, and P.~Tiep,
  eds.), Developments in Math., vol.~31, Springer, 2013, pp.~207--232.

\bibitem[RS11]{MR2842085}
Kristian Ranestad and Frank-Olaf Schreyer, \emph{On the rank of a symmetric
  form}, J. Algebra \textbf{346} (2011), 340--342. \MR{2842085}

\bibitem[Syl51a]{Sylvester:1851kx}
J.J. Sylvester, \emph{An essay on canonical forms, supplement to a sketch of a
  memoir on elimination, transformation and canonical forms}, originally
  published by George Bell, Fleet Street, London, 1851. Paper 34 in {\it
  Mathematical Papers}, Vol.~1, Chelsea, New York, 1973, originally published
  by Cambridge University Press in 1904., 1851.

\bibitem[Syl51b]{Sylvester:1851wd}
\bysame, \emph{On a remarkable discovery in the theory of canonical forms and
  of hyperdeterminants}, originally published in Philosophical Magazine,
  vol.~2, 1851, pp. 391--410. Paper 41 in {\it Mathematical Papers}, Vol.~1,
  Chelsea, New York, 1973, originally published by Cambridge University Press
  in 1904., 1851.

\bibitem[Yer32]{MR1506892}
J.~Yerushalmy, \emph{On the {C}onfiguration of the {N}ine {B}ase {P}oints of a
  {P}encil of {E}quianharmonic {C}ubics}, Amer. J. Math. \textbf{54} (1932),
  no.~2, 279--284. \MR{MR1506892}

\end{thebibliography}

\end{document}